\documentclass[12pt,a4paper]{amsart}

\usepackage{amsmath,amssymb,mathtools}
\usepackage{capt-of}
\usepackage[margin=2cm]{geometry}
\usepackage{hyperref}
\usepackage{tikz}

\newtheorem{theorem}{Theorem}[section]
\newtheorem{lemma}[theorem]{Lemma}
\newtheorem{proposition}[theorem]{Proposition}
\newtheorem{corollary}[theorem]{Corollary}
\theoremstyle{definition}
\newtheorem{definition}[theorem]{Definition}
\newtheorem{example}[theorem]{Example}
\newtheorem{remark}[theorem]{Remark}

\title[Mixed linkage of numerical semigroup ideals]
{Mixed linkage of relative ideals of numerical semigroups}

\author[I. Ojeda]{Ignacio Ojeda}
\address{Departamento de Matem\'aticas, Universidad de Extremadura,
06071 Badajoz, Spain}
\email{ojedamc@unex.es}
\urladdr{\url{https://orcid.org/0000-0003-3173-5934}}
\date{\today}


\hypersetup{
  pdftitle={Mixed linkage of relative ideals of numerical semigroups},
  pdfauthor={Ignacio Ojeda},
  pdfkeywords={numerical semigroup, relative ideal, mixed linkage,
    canonical ideal, duality, reflexive ideal}
}

\begin{document}

\begin{abstract}
We introduce principal and canonical linkage for relative ideals of a
numerical semigroup $S$ and study their interaction modulo translation.
On normalized $S$-reflexive ideals, a principal link followed by a
canonical link acts by addition of the canonical ideal $K$. The resulting
mixed components have paths as their underlying simple graphs, and their
even liaison classes are chains under inclusion. If $r(K)$ is the positive
reduction number of $K$, every even class has at most $r(K)$ elements and
every mixed component has at most $2r(K)$ vertices. For components with
a loop, the bound improves to $r(K)$.

For semigroups of type two, we compute the even class of a given ideal
from its occupied and gap blocks along arithmetic progressions and prove
that the largest even class has exactly $r(K)$ elements. This equality
can fail in type three. Both component bounds are sharp, and explicit
families of embedding dimension three give components of unbounded size
with and without loops.
\end{abstract}

\maketitle

\section{Introduction}

Duality and reduction provide two ways to study ideals: residual dualities reverse inclusion, while reduction records when repeated products stabilize. Numerical semigroup rings give a combinatorial setting in which these constructions can be studied together. Monomial fractional ideals are encoded by subsets of $\mathbb Z$, ideal quotients become residuals, and products become sums. We study the even liaison classes obtained by alternating principal
and canonical residual dualities. The reduction number of the
canonical ideal bounds their sizes, and this bound is attained for
every numerical semigroup of type two.

A \emph{numerical semigroup} is an additive submonoid $S\subseteq\mathbb N=\{0,1,2,\ldots\}$ with finite complement. Throughout the paper, $S$ denotes a numerical semigroup other than $\mathbb N$. We write $m=\min(S\setminus\{0\})$ for its \emph{multiplicity} and $F=\operatorname{F}(S)=\max(\mathbb Z\setminus S)$ for its \emph{Frobenius number}. Its \emph{embedding dimension} is the cardinality of its minimal generating set; $S=\langle n_1,\ldots,n_e\rangle$ denotes the semigroup generated by $n_1,\ldots,n_e$. 

For subsets $A,B\subseteq\mathbb Z$, let
$A+B=\{a+b:a\in A,\ b\in B\}$. For an integer $u$, let
$A+u=\{a+u:a\in A\}$.

A \emph{relative ideal} is a nonempty, bounded-below subset $I\subseteq\mathbb Z$ with $I+S\subseteq I$. Translation classes have unique \emph{normalized representatives} $I-\min I$, all of which contain $S$ and lie in $\mathbb N$. We refer to \cite{RG09} for the general theory.

For relative ideals $H,I,J$, the \emph{residual} of $I$ in $H$ is
\[
H-I=\{z\in\mathbb Z:z+I\subseteq H\}.
\]
For a field $\Bbbk$, set $A=\Bbbk[[S]]$ and associate to $I$ the monomial fractional ideal $\mathfrak a_I=(t^i:i\in I)A$. Membership in a monomial ideal is determined by the exponents of its terms, so
\[
\mathfrak a_H:\mathfrak a_I=\mathfrak a_{H-I},\qquad
\mathfrak a_I\mathfrak a_J=\mathfrak a_{I+J},\qquad
\mathfrak a_{I+u}=t^u\mathfrak a_I.
\]
The \emph{normalized canonical ideal} is
\[
K=K(S)=\{z\in\mathbb Z:F-z\notin S\}.
\]
A numerical semigroup $S$ is \emph{symmetric} if $K=S$, and
\emph{nonsymmetric} otherwise.

The ideal $\mathfrak a_K$ is a canonical fractional ideal of $A$. We study reciprocal residual relations
\[
J=H-I,\qquad I=H-J,
\]
with $H$ a translate of $S$ or of $K$. We call these \emph{principal} and \emph{canonical} links, respectively; ``principal'' describes the linking ideal. This terminology is motivated by the reciprocal quotients in classical liaison theory \cite[Definition~17(ii)]{MN02}; the residual relations above are the definition used here.

The normalized ideals form a finite graph with principal and canonical
links as edges, retaining loops and distinguishing the two edge types.
Its connected components are the \emph{mixed liaison classes}, and
vertices joined by walks of even length form the \emph{mixed even
liaison classes}. For either linking family alone, the classes have at
most two vertices and the even classes are singletons
(Proposition~\ref{prop:pure-linkage}(b)). If $S$ is nonsymmetric, the
mixed component $\{S,K\}$ is a single even class: the principal loop
at $S$ followed by the canonical edge from $S$ to $K$ gives a walk of
length two (Proposition~\ref{prop:class-S}).

The normalized ideals also form the ideal class monoid studied in
\cite{BK16,CDG24}. The interaction of its addition with residual
duality is expressed by the known identities
\[
K-(S-I)=I+K,\qquad S-(S-I)=(I+K)-K;
\]
see \cite[Remark~9 and Lemma~16]{CDG26}. An ideal is
\emph{$S$-reflexive} when $S-(S-I)=I$. This condition determines
when the principal residual is reciprocal. A principal link followed
by a canonical link gives the injective partial map
\[
\mathsf T(I)=I+K,\qquad I\text{ normalized and }S\text{-reflexive}.
\]
Thus additive iteration describes even linkage only for as long as
reflexivity permits the next step.

The \emph{positive reduction number} of $K$ is
\[
r(K)=\min\{n\ge1:nK=(n+1)K\}.
\]
This invariant is studied in \cite[Section~3]{BS19}, and the bound
$r(K)\le m-1$ follows from \cite[Corollary~4.8]{BK16}.
Our first result relates $r(K)$ to the structure of mixed linkage.

\begin{theorem}\label{thm:main-structure}
With the above notation, the following statements hold.
\begin{enumerate}
\item[(a)] Every mixed even liaison class is totally ordered by inclusion
      and has at most $r(K)\le m-1$ elements.
\item[(b)] After removing loops and identifying edges with the same
      endpoints, every mixed component is a finite path with at most
      $2r(K)$ vertices. If the original component has a loop, it has
      at most $r(K)$ vertices.
\item[(c)] The Frobenius number of normalized representatives is constant
      on each mixed component.
\end{enumerate}
\end{theorem}

Theorem~\ref{thm:main-structure} is proved in
Section~\ref{sec:reduction-bounds}. Away from the fixed points of
$\mathsf T$, the mixed even liaison classes are maximal chains under
iteration of $\mathsf T$, and canonical duality reverses their order
(Proposition~\ref{prop:partial-dynamics}).

A mixed component is recovered from any of its even liaison classes
by adjoining the normalized canonical duals of its ideals
(Theorem~\ref{thm:paths}).

\medskip
The \emph{pseudo-Frobenius numbers} of $S$ are the integers $f\notin S$ such that $f+(S\setminus\{0\})\subseteq S$. Their set is denoted by $\operatorname{PF}(S)$, and its cardinality is the \emph{type} of $S$.

The canonical generating set \cite[Corollary~4.11]{CDG24} gives
\[
I+K=\bigcup_{f\in\operatorname{PF}(S)}\bigl(I+(F-f)\bigr).
\]
Thus the type counts the translates in a single addition of $K$,
whereas $r(K)$ bounds the number of ideals in a mixed even liaison
class. Attaining this bound requires a starting ideal for which
the first $r(K)-1$ steps are defined and strict. In type two,
there is just one nonzero shift, and we use the resulting block
structure to construct such an ideal.

\begin{theorem}\label{thm:type-two-attainment}
Let $S$ be a numerical semigroup of type two. Write
$K=S\cup(d+S)$ with $d>0$. Then the largest
mixed even liaison class has cardinality $r(K)$. This bound is
attained by the normalized relative ideal
\begin{equation}\label{eq:type-two-attaining-ideal}
H=\mathbb N\setminus\{jd-s:1\le j\le r(K),\ s\in S\},
\end{equation}
whose mixed even liaison class is
$\{H,H+K,\ldots,H+(r(K)-1)K\}$.
\end{theorem}

Theorem~\ref{thm:type-two-attainment} is proved in
Section~\ref{sec:type-two-attainment}.

\medskip
For a semigroup $S$ of type two and a normalized ideal $I$ with $I+K\ne I$, let $a(I)$ and $b(I)$ be the minimum lengths of a finite occupied block and a bounded gap block, respectively, over all arithmetic progressions of step $d$, where $d$ is the shift in Theorem~\ref{thm:type-two-attainment}. Then the mixed even liaison class of $I$ has exactly $a(I)+b(I)-1$ elements (Theorem~\ref{thm:type-two-classes}). If $I+K=I$, its mixed even liaison class is $\{I\}$.

Every nonsymmetric numerical semigroup of embedding dimension three
has type two, so these results apply to all such semigroups.
For symmetric semigroups, every mixed even liaison class is a singleton.
The semigroups $\langle5,7,8\rangle$ and $\langle7,9,10,11\rangle$
both have $r(K)=3$, but their largest mixed even liaison classes have
three and two elements, respectively
(Examples~\ref{ex:sharp-bound} and~\ref{ex:strict-bound}).

The component bound $2r(K)$ is attained over $\langle5,7,8\rangle$
(Example~\ref{ex:sharp-bound}). For odd $m\ge5$, the family
$\langle m,m+1,m+2\rangle$ attains the bound $r(K)=m-1$ for
components with a loop (Theorem~\ref{thm:loop-family}).
In Theorem~\ref{thm:unbounded}, we give, for each odd $m\ge5$,
a loopless component with $m+1$ vertices over
$\langle m,m+2,m+4\rangle$. Its two even liaison classes each have
$(m+1)/2$ elements. Thus the sizes of mixed components and of mixed
even liaison classes are unbounded already in embedding dimension
three, the least dimension in which this can occur.

Finally, for nonsymmetric semigroups, almost symmetry is equivalent
to $\{S,K\}$ being the only nonsingleton mixed even liaison class
(Proposition~\ref{prop:almost-symmetric}). This follows from the
reflexivity characterization in \cite[Proposition~53]{CDG26}.
For the ring-theoretic counterpart, see \cite[Proposition~4.4]{CDS26}.

\section{Relative ideals, dualities, and linkage}

Every relative ideal $I$ contains $i+F+1+\mathbb N$ for each $i\in I$.
Its Frobenius number $\operatorname{F}(I)=\max(\mathbb Z\setminus I)$
therefore exists and satisfies
$\operatorname{F}(I+u)=\operatorname{F}(I)+u$.

\begin{definition}
Relative ideals $I$ and $J$ are \emph{translation equivalent} if
$I=J+u$ for some $u\in\mathbb Z$. The \emph{normalization} of $I$
is $I^\mathrm{nor}=I-\min I$. An ideal is \emph{normalized} if its
minimum is zero.
\end{definition}

Each translation class has a unique normalized representative,
and there are only finitely many normalized ideals
\cite[Lemma~2.2 and Proposition~2.4]{BK16}.

Subtraction of a relative ideal denotes the residual operation.

\begin{lemma}\label{lem:colon-basic}
Let $H,I,J$ be relative ideals of $S$ and let $a,b\in\mathbb Z$.
\begin{enumerate}
\item[(a)] $H-I$ is a relative ideal.
\item[(b)] If $I\subseteq J$, then $H-J\subseteq H-I$.
\item[(c)] $(H+a)-(I+b)=(H-I)+(a-b)$.
\item[(d)] $(H-I)-J=H-(I+J)$.
\item[(e)] $I\subseteq H-(H-I)$.
\item[(f)] $H-(H-(H-I))=H-I$.
\end{enumerate}
\end{lemma}

\begin{proof}
For (a), $\operatorname{F}(H)+1-\min I$ belongs to $H-I$, and
every $z\in H-I$ satisfies $z+\min I\in H$. Hence $H-I$ is nonempty
and bounded below. It is $S$-stable because $H$ is $S$-stable.

Parts (b)--(d) follow directly from the definition. For (e), every
$i\in I$ satisfies $i+(H-I)\subseteq H$. For (f), write $D_H(I)=H-I$.
By (e), $I\subseteq D_H^2(I)$, so (b) gives
$D_H^3(I)\subseteq D_H(I)$. Applying (e) to $D_H(I)$ gives the
reverse inclusion.
\end{proof}

\begin{definition}
Let $L$ be a relative ideal of $S$. Relative ideals $I$ and $J$ are
\emph{directly linked by $L$} if
\[
J=L-I\qquad\text{and}\qquad I=L-J.
\]
Let $\mathcal L$ be a nonempty family of relative ideals of $S$ closed
under integer translations. The ideals $I$ and $J$ are
\emph{directly $\mathcal L$-linked} if they are directly linked by some
$L\in\mathcal L$. A chain of direct $\mathcal L$-links defines
\emph{$\mathcal L$-liaison}, and such a chain of even length defines
\emph{even $\mathcal L$-liaison}. Chains may have length zero or
revisit an ideal.
\end{definition}

\begin{proposition}\label{prop:translation-link}
If $I,J$ are directly linked by $L$, then $I+u,J+v$ are directly
linked by $L+u+v$ for every $u,v\in\mathbb Z$. In particular, their
normalized representatives are directly linked by
$L-\min I-\min J$.
\end{proposition}

\begin{proof}
By Lemma~\ref{lem:colon-basic}(c),
$(L+u+v)-(I+u)=(L-I)+v=J+v$. Interchanging $(I,u)$ and $(J,v)$
gives the other equality.
\end{proof}

\begin{definition}
For a family $\mathcal L$ as above, the \emph{normalized liaison graph}
has normalized ideals as vertices. An edge joins two vertices if they
are directly $\mathcal L$-linked, with a loop for a direct self-link.
Its connected components are the \emph{liaison classes modulo
translation}. Its \emph{even liaison classes modulo translation}
are the equivalence classes for the relation of being joined by a
walk of even length.
\end{definition}

Proposition~\ref{prop:translation-link} and translation closure make
adjacency independent of the chosen representatives. Walks of length
zero, reversal and concatenation give the even-walk equivalence
relation. Hereafter, liaison classes are taken modulo translation
unless otherwise specified.

Write $I^*=S-I$ and $I^K=K-I$. The condition $I^{**}=I$ is
$S$-reflexivity. It is invariant under translation by
Lemma~\ref{lem:colon-basic}(c), and every $S$-dual is reflexive by
Lemma~\ref{lem:colon-basic}(f). See \cite{HWRS06,RG09} for
$S$-duality.

The definition of $K$ gives $0\in K\subseteq\mathbb N$. Since $K$
is $S$-stable, it also contains $S$. We use the following reflection
formula for canonical duality \cite{RG09,DS21}.

\begin{lemma}
For every relative ideal $I$,
\begin{equation}\label{eq:reflection}
K-I=\{F-x:x\in\mathbb Z\setminus I\}.
\end{equation}
Consequently, $K-(K-I)=I$, $K-S=K$ and $K-K=S$.
\end{lemma}

\begin{proof}
The condition $z\notin K-I$ means that $F-z-i\in S$ for some
$i\in I$. This implies $F-z\in I$ by $S$-stability. Conversely,
if $F-z\in I$, take $i=F-z$ to obtain $F-z-i=0\in S$.
This proves \eqref{eq:reflection}, and applying it twice gives
biduality. Since $0\in S$ and $K$ is $S$-stable, $K-S=K$.
Biduality for $I=S$ then gives $K-K=S$.
\end{proof}

Define the principal and canonical linking families by
\[
\mathcal L_\mathrm{pr}(S)=\{a+S:a\in\mathbb Z\},\qquad
\mathcal L_\mathrm{can}(S)=\{a+K:a\in\mathbb Z\}.
\]

\begin{proposition}\label{prop:pure-linkage}
With the above notation, the following statements hold.
\begin{enumerate}
\item[(a)] A normalized ideal $I$ has a principal neighbor exactly
when it is $S$-reflexive. Its principal neighbor is then
$(I^*)^\mathrm{nor}$. Its canonical neighbor always exists and is
$(I^K)^\mathrm{nor}$.
\item[(b)] Each principal or canonical liaison class has at most two
vertices, and each even class for either family is a singleton.
\item[(c)] The two linking families coincide if and only if $S$ is
symmetric, equivalently if and only if $S$ has type one.
\end{enumerate}
\end{proposition}

\begin{proof}
For (a), the residual of $I$ in $a+S$ is $a+I^*$, and the
reciprocal residual is $(a+S)-(a+I^*)=I^{**}$. For $a+K$, the
same calculation gives $a+I^K$ and $K-(K-I)=I$. The normalized
neighbors are therefore independent of $a$. For (b), each vertex
has at most one neighbor of a given type, so two successive links
of that type return to the starting vertex.

For (c), if the families coincide, then $S=a+K$ for some
$a\in\mathbb Z$. Since $\min S=\min K=0$, we have $a=0$ and
$S=K$. Conversely, $S=K$ gives identical families. The equivalence
with type one is standard \cite{RG09}.
\end{proof}

For canonical or mixed linkage, two relative ideals are joined by an
even chain of links if and only if their normalized representatives
belong to the same even liaison class in the normalized graph.
By Proposition~\ref{prop:translation-link}, normalization preserves
the length of a chain because links whose endpoints normalize to the
same ideal become loops. Conversely, canonical biduality and the
translation rule give
\[
I \xleftrightarrow{\;K\;} K-I
  \xleftrightarrow{\;K+u\;} I+u
\qquad (u\in\mathbb Z).
\]
Thus each ideal is joined to its normalized representative by two
canonical links. Adding two such links at each end of an even walk
between the normalized representatives gives an even chain between
the original ideals.

\begin{example}
Let $S=\langle5,6,8\rangle=\{0,5,6,8\}\cup(10+\mathbb N)$ and
$I=\{0,1,3\}+S=\mathbb N\setminus\{2,4\}$. Then $F=9$ and
$K=S\cup(2+S)$. Put
\[
U=\mathbb N\setminus\{1,3,4\},\qquad
V=\mathbb N\setminus\{1,2,3,4\},\qquad
W=\mathbb N\setminus\{4\}.
\]
Computing residuals gives
\[
K-I=5+U,\qquad S-I=5+V,\qquad S-V=5+I,\qquad K-V=5+W.
\]
Since $U^{**}=U\cup\{4\}$ and $W^{**}=\mathbb N$, both endpoints
are nonreflexive. Thus the full mixed component of $I$ is the following
path, with canonical links solid and principal links dashed.
\begin{center}
\begin{tikzpicture}[
  vertex/.style={circle,draw,minimum size=8mm,inner sep=1pt}]
\node[vertex] (U) at (0,0) {$U$};
\node[vertex] (I) at (2.2,0) {$I$};
\node[vertex] (V) at (4.4,0) {$V$};
\node[vertex] (W) at (6.6,0) {$W$};
\draw[thick] (U) -- node[above] {$\mathrm{can}$} (I);
\draw[thick,dashed] (I) -- node[above] {$\mathrm{pr}$} (V);
\draw[thick] (V) -- node[above] {$\mathrm{can}$} (W);
\end{tikzpicture}
\end{center}
Its mixed even liaison classes are $\{U,V\}$ and $\{I,W\}$.
\end{example}

\section{Mixed linkage}

Let
$\mathcal L_\mathrm{mix}(S)=\mathcal L_\mathrm{pr}(S)\cup\mathcal L_\mathrm{can}(S)$.
In the normalized graph we retain loops and distinguish principal
and canonical edges, even when they have the same endpoints.
For a normalized ideal $I$, write
\[
\mathsf C(I)=(K-I)^\mathrm{nor},\qquad \mathsf P(I)=(S-I)^\mathrm{nor}.
\]
Let $\mathcal R$ be the set of normalized $S$-reflexive ideals.
Proposition~\ref{prop:pure-linkage}(a) identifies $\mathsf C(I)$
as the canonical neighbor of $I$ and $\mathsf P(I)$ as its
principal neighbor when $I\in\mathcal R$. A loop counts as a
neighbor equal to the vertex itself. By reciprocity, $\mathsf C$
and $\mathsf P|_{\mathcal R}$ are involutions on their respective
domains.

\subsection{Composition and admissible iteration}

\begin{lemma}\label{lem:mixed-identities}
For every relative ideal $I$,
\begin{equation}\label{eq:mixed-identities}
S-I=K-(I+K),\qquad K-(S-I)=I+K,\qquad
S-(K-I)=I-K,
\end{equation}
and
\begin{equation}\label{eq:closure}
I^{**}=(I+K)-K.
\end{equation}
For every normalized ideal $I$,
\begin{equation}\label{eq:normalized-composition}
\mathsf C(\mathsf P(I))=I+K.
\end{equation}
If $I$ is $S$-reflexive, this composition is realized by a principal
link followed by a canonical link.
\end{lemma}

\begin{proof}
Since $S=K-K$, residuation gives
$S-I=(K-K)-I=K-(I+K)$.
Canonical biduality yields $K-(S-I)=I+K$. Also,
\[
S-(K-I) = K-(K+(K-I)) = (K-(K-I))-K=I-K.
\]
Replacing $I$ by $I+K$ gives \eqref{eq:closure}.

For normalized $I$, let $b=\min(S-I)$. Then
\[
K-\mathsf P(I)=K-((S-I)-b)=(I+K)+b.
\]
Since $\min(I+K)=0$, normalizing gives \eqref{eq:normalized-composition}.
\end{proof}

The identities in \eqref{eq:mixed-identities} and \eqref{eq:closure}
are standard consequences of residuation and canonical biduality.
In particular, $K-(S-I)=I+K$ and the reflexive-closure formula also
appear in \cite[Remark~9 and Lemma~16]{CDG26}. Starting from a
normalized ideal $I$, successive principal--canonical pairs give
$I,I+K,I+2K,\ldots$ as long as the ideal at the start of each pair
is $S$-reflexive. This gives the partial map
\[
\mathsf T=\mathsf C\circ(\mathsf P|_{\mathcal R})
  \colon\mathcal R\longrightarrow\mathsf C(\mathcal R),
\qquad \mathsf T(I)=I+K.
\]
We put $\mathcal N=\{I\text{ normalized}:I+K=I\}$ and call its
elements \emph{$K$-stable}. A composition of partial maps is defined
wherever all its successive steps are defined. Equality of partial
maps includes equality of their domains.

\begin{proposition}\label{prop:partial-dynamics}
The following statements hold.
\begin{enumerate}
\item[(a)] The map $\mathsf T$ is a bijection from $\mathcal R$ onto
$\mathsf C(\mathcal R)$ and satisfies
$\mathsf C\mathsf T\mathsf C=\mathsf T^{-1}$.

\item[(b)] The fixed points of $\mathsf T$ are precisely the ideals
in $\mathcal N$. Each such ideal forms a singleton mixed even liaison
class, and $\mathsf C$ permutes the elements of $\mathcal N$.

\item[(c)] Every mixed even liaison class disjoint from $\mathcal N$
is a maximal chain under iteration of $\mathsf T$:
\[
I_0\subsetneq\cdots\subsetneq I_{\ell-1},\qquad
I_{j+1}=\mathsf T(I_j)=I_j+K\quad(0\le j<\ell-1),
\]
where $\ell\ge1$. Its last ideal $I_{\ell-1}$ is its only nonreflexive
member. Canonical duality sends the chain to a chain of the same
length in reverse order.
\end{enumerate}
\end{proposition}

\begin{proof}
(a) Since $\mathsf P|_{\mathcal R}$ and $\mathsf C$ are involutions
on their respective domains, $\mathsf T$ is a bijection from
$\mathcal R$ onto $\mathsf C(\mathcal R)$. Both
$(\mathsf P|_{\mathcal R})\circ\mathsf C$ and
$\mathsf C\mathsf T\mathsf C$ have domain $\mathsf C(\mathcal R)$
and agree there with $\mathsf T^{-1}$.

(b) If $I\in\mathcal N$, then $I+K=I$ and $0\in K$ give $I-K=I$.
By \eqref{eq:closure}, $I^{**}=I$, so $I\in\mathcal R$ and
$\mathsf T(I)=I$. Conversely, every fixed point of $\mathsf T$
satisfies $I+K=I$. The conjugacy in (a) shows that $\mathsf C$
permutes these fixed points.

In an even walk, two consecutive links of the same type form a
round trip and can be deleted without changing the endpoints or
the parity. If the reduced walk is nonempty, it alternates the
two types. Grouping its links into pairs gives a defined iterate
of $\mathsf T$ or $\mathsf T^{-1}$. Conversely, every such iterate
is realized by an even walk. Since both maps fix each ideal in
$\mathcal N$, its mixed even liaison class is a singleton.

(c) By (b), the remaining mixed even liaison classes are the
components generated by the steps of $\mathsf T$ and
$\mathsf T^{-1}$ outside $\mathcal N$. Every defined forward step
there is a strict inclusion. Since $\mathsf T$ is injective and
there are finitely many normalized ideals, each such component is
a maximal finite chain. Every member except the last belongs to
$\mathcal R$, and the last lies outside $\mathcal R$ by maximality.
The conjugacy in (a) sends forward steps to backward steps, so
$\mathsf C$ sends each maximal chain to one of the same length
in reverse order.
\end{proof}

\subsection{Component structure and Frobenius number}

\begin{theorem}\label{thm:paths}
For a fixed numerical semigroup $S$, every mixed liaison class is finite. After removing loops and identifying principal and canonical edges with the same endpoints, its graph is a path, possibly consisting of a single vertex.
For every mixed even liaison class $\mathcal E$, the vertex set of
its component is $\mathcal E\cup\mathsf C(\mathcal E)$, where
$\mathsf C(\mathcal E)=\{\mathsf C(I):I\in\mathcal E\}$.
More precisely, each component has one of the following forms.
\begin{enumerate}
\item[(a)] A component contained in $\mathcal N$ consists either of one vertex with both a principal and a canonical loop, or of two vertices joined by coincident principal and canonical edges. Its even classes are singletons.
\item[(b)] A component disjoint from $\mathcal N$ and without loops
has $2\ell$ vertices for some $\ell\ge1$, two even classes of $\ell$
vertices each, and exactly two nonreflexive vertices, which are its
endpoints. The two even classes are the parts of the path's bipartition.
\item[(c)] A component disjoint from $\mathcal N$ and with a loop has
a single even class, exactly one nonreflexive vertex, and exactly one
loop. If it has $\ell$ vertices, this loop is canonical when $\ell$
is odd and principal when $\ell$ is even. When $\ell>1$, the endpoints
of the underlying path are the nonreflexive vertex and the vertex
carrying the loop.
\end{enumerate}
\end{theorem}

\begin{proof}
Fix $I\in\mathcal E$ and choose a walk from $I$ to a vertex $J$ in
its mixed component. If the walk has even length, then
$J\in\mathcal E$. If it has odd length, appending the canonical
link from $J$ to $\mathsf C(J)$ gives $\mathsf C(J)\in\mathcal E$,
so $J\in\mathsf C(\mathcal E)$ by involutivity. Conversely, each
vertex of $\mathsf C(\mathcal E)$ is joined by a canonical edge
to a vertex of $\mathcal E$.

(a) If $I\in\mathcal N$, then $\mathsf T(I)=I$, so
$\mathsf P(I)=\mathsf C(I)$. By
Proposition~\ref{prop:partial-dynamics}(b), $\mathsf C$ preserves
$\mathcal N$ and the even class of $I$ is a singleton. Its mixed
component is therefore $\{I,\mathsf C(I)\}$ and lies in $\mathcal N$.
At each vertex the two types of link give the same neighbor, which
gives the two possibilities in (a).

For the remaining cases, let $\mathcal E=(I_0,\ldots,I_{\ell-1})$
be a chain disjoint from $\mathcal N$. By
Proposition~\ref{prop:partial-dynamics}(c), its canonical image
is a chain of the same length in reverse order. The two chains
are either equal or disjoint. For $0\le j<\ell-1$, we have
$\mathsf P(I_j)=\mathsf C(I_{j+1})$. At every reflexive vertex of
the component, the principal and canonical neighbors are distinct,
since their equality would give a fixed point of $\mathsf T$.

(b) If the two chains are disjoint, every edge joins a vertex of
one chain to a vertex of the other. The two nonreflexive vertices
have degree one in the underlying simple graph, and all other
vertices have degree two. The graph is therefore a path with
$2\ell$ vertices. Its two even classes are the two chains, which
are the parts of its bipartition.

(c) If $\mathsf C(\mathcal E)=\mathcal E$, the component is the
single even class $\mathcal E$. Reversal of the chain gives
\[
\mathsf C(I_j)=I_{\ell-1-j}\quad(0\le j<\ell),\qquad
\mathsf P(I_j)=I_{\ell-2-j}\quad(0\le j<\ell-1).
\]
There is exactly one loop. It is canonical at $I_{(\ell-1)/2}$
when $\ell$ is odd and principal at $I_{(\ell-2)/2}$ when $\ell$
is even. The only nonreflexive vertex is $I_{\ell-1}$. If $\ell>1$,
this vertex and the vertex carrying the loop have degree one in
the underlying simple graph, and all others have degree two.
Thus they are the endpoints of the path. For $\ell=1$, the
underlying path consists of a single vertex.
\end{proof}

The reflection formula controls normalization of a canonical residual.

\begin{lemma}\label{lem:conductor-dual}
Let $I$ be a normalized relative ideal and let $J=\mathsf C(I)$. Then
$\min(K-I)=F-\operatorname{F}(I)$, $J=\{\operatorname{F}(I)-x:x\in\mathbb Z\setminus I\}$ and $\operatorname{F}(J)=\operatorname{F}(I)$.
\end{lemma}

\begin{proof}
The least element in the reflection \eqref{eq:reflection} is $F-\operatorname{F}(I)$, which proves the first two formulas. Since $I$ has minimum zero, $\operatorname{F}(I)$ is absent from $J$, and every integer larger than $\operatorname{F}(I)$ belongs to $J$. This also includes the case $I=\mathbb N$, with $\operatorname{F}(I)=-1$ and $J=\mathbb N$.
\end{proof}

The Frobenius number is constant on every mixed liaison class of normalized ideals.

\begin{proposition}\label{prop:conductor-invariant}
If normalized ideals $I,J$ are directly linked by a principal or a canonical ideal, then
\[
\operatorname{F}(I)=\operatorname{F}(J).
\]
 Moreover, if $I$ is normalized and $S$-reflexive, then
$\min(S-I)=F-\operatorname{F}(I)$ and $\operatorname{F}(I+K)=\operatorname{F}(I)$.
\end{proposition}

\begin{proof}
The canonical assertion is Lemma~\ref{lem:conductor-dual}. For the principal assertion, let $I$ be normalized and reflexive, put $D=S-I$, and write $f=\operatorname{F}(I)$ and $b=\min D$. If $z<F-f$, then $F-z>f$, so $F-z\in I$. Consequently $z\notin D$, since $z+(F-z)=F\notin S$. It follows that $b\ge F-f$.

If $b>F-f$, then $f+d>F$ for every $d\in D$, so $f+D\subseteq S$. This would imply $f\in S-D=I$, contradicting the definition of $f$. Hence $b=F-f$.

Because $0\in I$, we have $D\subseteq S$, so $F\notin D$. Every integer greater than $F$ belongs to $D$, since $I\subseteq\mathbb N$. Thus $\operatorname{F}(D)=F$, and normalization gives $
\operatorname{F}(\mathsf P(I))=\operatorname{F}(D)-b=F-b=f$.
This proves invariance under a principal link. Finally, $I+K=\mathsf C(\mathsf P(I))$ by Lemma~\ref{lem:mixed-identities}, so invariance under the two successive links gives $\operatorname{F}(I+K)=f$.
\end{proof}

The minimum formula also follows from \cite[Lemma~21]{CDG26}
applied to the translate $(F-\operatorname{F}(I))+I$, which has
the same conductor as $S$. Writing $c(H)=\operatorname{F}(H)+1$,
the formula becomes $c(I)=c(S)-\min(S-I)$. If $S$ is symmetric,
this holds for every normalized ideal. The two-generated case
appears in \cite[Corollary~3.6]{AM21}.

\subsection{Reduction bounds}\label{sec:reduction-bounds}

\begin{corollary}\label{cor:component-bound}
Every mixed even liaison class is totally ordered by inclusion and has at most $r(K) \le m-1$ elements. Consequently, every mixed component has at most
\[
2r(K)\le2(m-1)
\]
vertices. A component with a loop has at most $r(K)$ vertices.
\end{corollary}

\begin{proof}
The total order follows from Proposition~\ref{prop:partial-dynamics}. A strict chain with more than $r(K)$ vertices would admit $r(K)$ steps from its first ideal $I$ to $J=I+r(K)K$. Since $r(K)K=(r(K)+1)K$, we have $J+K=J$, so $\mathsf T(J)=J$. Injectivity rules out a predecessor distinct from $J$, contradicting strictness. The fixed singleton classes also satisfy the bound because $r(K)\ge1$. By Theorem~\ref{thm:paths}, every component has at most two even classes, and a component with a loop has exactly one. The component bounds follow.
\end{proof}

\begin{proof}[Proof of Theorem~\ref{thm:main-structure}]
Corollary~\ref{cor:component-bound} gives the inclusion order on
each mixed even liaison class and all the cardinality bounds.
Theorem~\ref{thm:paths} gives the path structure, and
Proposition~\ref{prop:conductor-invariant} gives the constancy of
the Frobenius number on each mixed component.
\end{proof}

\begin{example}\label{ex:sharp-bound}
The bound $2r(K)$ is attained for $S=\langle5,7,8\rangle$.
The least elements of $S$ in the residue classes $0,\ldots,4$
modulo $5$ are $0,16,7,8,14$, respectively. Hence $F=11$ and
$K=\{0,2\}+S$. We have $4\in2K\setminus K$,
$6\in3K\setminus2K$, and $3K=4K$, the last equality following
from $8\in S$. Thus $r(K)=3$.

Consider the following six normalized ideals.
\[
\begin{array}{c|l}
\text{Ideal} & \text{Description}\\ \hline
A_0 & \mathbb N\setminus\{1,6\}\\
B_0 & \{0,5\}\cup(7+\mathbb N)\\
A_1 & \mathbb N\setminus\{1,4,6\}\\
B_1 & \{0,2,5\}\cup(7+\mathbb N)\\
A_2 & \mathbb N\setminus\{1,2,4,6\}\\
B_2 & \{0,2,4,5\}\cup(7+\mathbb N).
\end{array}
\]
Each set contains $S$ and $7+\mathbb N$. The sum of a positive
member and a positive element of $S$ is at least $7$, so these
sets are relative ideals. Their residuals satisfy
\[
\begin{aligned}
K-A_j&=5+B_j,     & K-B_j&=5+A_j &&(0\le j\le2),\\
S-B_j&=5+A_{j+1}, & S-A_{j+1}&=5+B_j &&(0\le j<2).
\end{aligned}
\]
The linking ideals $K-5$ and $S-5$ therefore give the path shown
in Figure~\ref{fig:sharp-bound}. Its endpoints satisfy
\[
A_0^{**}=\mathbb N\setminus\{1\}\ne A_0,\qquad
B_2^{**}=B_2\cup\{6\}\ne B_2,
\]
so they admit no principal link. The residual equalities determine
all other neighbors and give no loops. The full component has six
vertices, attaining $2r(K)=6$. Its two even classes are
$\{A_0,A_1,A_2\}$ and $\{B_0,B_1,B_2\}$, each with $r(K)=3$
elements. All six normalized ideals have Frobenius number $6$.

\begin{center}
\begin{minipage}{\linewidth}
\centering
\begin{tikzpicture}[
  every node/.style={font=\small},
  vertex/.style={circle,draw,minimum size=9mm,inner sep=1pt},
  endpoint/.style={vertex,double,double distance=1.2pt}]
\node[endpoint] (A0) at (0,0) {$A_0$};
\node[vertex]   (A1) at (2.3,0) {$A_1$};
\node[vertex]   (A2) at (4.6,0) {$A_2$};
\node[vertex]   (B0) at (0,-1.8) {$B_0$};
\node[vertex]   (B1) at (2.3,-1.8) {$B_1$};
\node[endpoint] (B2) at (4.6,-1.8) {$B_2$};

\node[anchor=east] at (-0.7,0) {First even class};
\node[anchor=east] at (-0.7,-1.8) {Second even class};

\draw[thick] (A0) -- node[right] {$\mathrm{can}$} (B0);
\draw[thick] (A1) -- node[right] {$\mathrm{can}$} (B1);
\draw[thick] (A2) -- node[right] {$\mathrm{can}$} (B2);
\draw[thick,dashed] (B0) -- node[below,sloped] {$\mathrm{pr}$} (A1);
\draw[thick,dashed] (B1) -- node[below,sloped] {$\mathrm{pr}$} (A2);
\end{tikzpicture}
\captionof{figure}{The six-vertex component. Solid edges are canonical and dashed edges are principal. The rows are the two even liaison classes. Double circles mark the nonreflexive endpoints $A_0$ and $B_2$.}
\label{fig:sharp-bound}
\end{minipage}
\end{center}
\end{example}

\begin{remark}
The least positive integer $n$ such that $\mathsf T^n=\mathrm{id}_{\mathcal N}$ is exactly $\max_{\mathcal E}|\mathcal E|$, where $\mathcal E$ runs through the mixed even liaison classes. Indeed, on a strict chain with $\ell$ vertices, $\mathsf T^n$ is defined at some vertex precisely when $n<\ell$, whereas every power fixes $\mathcal N$. This least exponent can be strictly smaller than $r(K)$; see Example~\ref{ex:strict-bound}.
\end{remark}

For the ring interpretation, take $A=\Bbbk[[S]]$ with $\Bbbk$ infinite and $\Omega=(t^u:u\in K)A$. For sufficiently large $a\in S$, the canonical ideal $C=t^a\Omega\subseteq A$ satisfies
\[
C^{n+1}=t^aC^n\quad\Longleftrightarrow\quad(n+1)K=nK
\qquad(n\ge0),
\]
where $0K=S$. Hence $(t^a)$ is a minimal reduction of $C$, and $r(K(S))=\max\{1,\rho(A)\}$, where $\rho(A)$ is the \emph{canonical index} \cite[Proposition~4.1 and Definition~4.2]{GGHV17}. The positive convention differs from $\rho(A)$ only in the symmetric case, when $\rho(A)=0$ and $r(K)=1$.

\subsection{Symmetry and small components}

\begin{corollary}\label{cor:symmetric-classes}
Suppose that $S$ is symmetric. Every relative ideal is $S$-reflexive, and its principal, canonical, and mixed liaison classes coincide. For a normalized ideal $I$, this common class is $\{I,(S-I)^\mathrm{nor}\}$. Every mixed even liaison class is a singleton.
\end{corollary}

\begin{proof}
Since $K=S$, canonical biduality gives $S-(S-I)=I$ for every relative ideal $I$. Both types of link therefore act by the same involution $I\mapsto(S-I)^\mathrm{nor}$, giving the stated common class. Every even walk returns to its starting vertex.
\end{proof}

Every complete intersection numerical semigroup is symmetric, as is
every numerical semigroup of embedding dimension two \cite{RG09}.
Their mixed liaison classes therefore have at most two vertices.
In embedding dimension three, every nonsymmetric numerical semigroup
has type two \cite{RG09}, so its mixed liaison classes are described
in Section~\ref{sec:type-two}.

\begin{proposition}\label{prop:class-S}
If $S$ is nonsymmetric, the mixed liaison class of $S$ is $\{S,K\}$, and these two vertices lie in the same mixed even class.
\end{proposition}

\begin{proof}
We have $S^*=S$ and $\mathsf C(S)=K$, so $S$ has a principal loop and a canonical edge to $K$. Since $K\ne S$, the residual $S-K$ is contained in $S\setminus\{0\}$. Hence $F+(S-K)\subseteq S$, and therefore $F\in K^{**}$. But $F\notin K$, so $K$ is not $S$-reflexive and has no principal edge. Its canonical neighbor is $S$. These are all the edges in the component. The principal loop at $S$ gives an even walk from $S$ to $K$.
\end{proof}

Write $M=S\setminus\{0\}$. Recall that $S$ is \emph{almost symmetric} if $M+K=M$; see \cite{DS21,CDG26}.

\begin{proposition}\label{prop:almost-symmetric}
If $S$ is almost symmetric, every mixed component has at most two vertices. Moreover, for a nonsymmetric numerical semigroup $S$, the following are equivalent:
\begin{enumerate}
\item $S$ is almost symmetric.
\item The only mixed even liaison class with more than one vertex is $\{S,K\}$.
\end{enumerate}
\end{proposition}

\begin{proof}
If $S$ is symmetric, the component bound follows from Corollary~\ref{cor:symmetric-classes}. Suppose first that $S$ is almost symmetric and nonsymmetric. The reflexivity characterization underlying this implication is \cite[Proposition~53]{CDG26}; its ring-theoretic counterpart is \cite[Proposition~4.4]{CDS26}. We give a direct proof of the consequence used here.

Let $I\ne S$ be a normalized reflexive ideal and put $D=S-I$. Since $0\in I$, we have $D\subseteq S$, and $0\in D$ would force $I=S$. Thus $D\subseteq M$. Moreover, $I+D\subseteq S$ by definition of $D$, and $I\subseteq\mathbb N$ implies $I+D\subseteq M$. It follows that $(I+K)+D\subseteq M+K=M\subseteq S$, so $I+K\subseteq S-D=I$. The reverse inclusion holds because $0\in K$, hence $I+K=I$. Therefore $I\in\mathcal N$, and Theorem~\ref{thm:paths}(a) shows that its full component has at most two vertices and singleton even classes.

The component containing $S$ is $\{S,K\}$, and it is a single even class by Proposition~\ref{prop:class-S}. Every other component containing a reflexive vertex has already been treated. A remaining component has only nonreflexive vertices and hence only canonical edges; canonical duality gives at most two vertices and singleton even classes. This proves the component bound and (1)$\Rightarrow$(2).

Conversely, suppose that $S$ is nonsymmetric and (2) holds. The maximal ideal $M$ is reflexive: indeed, $0,F\in S-M$, so $M\subseteq M^{**}\subseteq S$ and $0\notin M^{**}$, which force $M^{**}=M$. Thus $A=M-m$ is a normalized reflexive ideal. We have $A\ne S$, because $M=m+S$ would force $S=\langle m\rangle$, contrary to cofiniteness and $m\ge2$. Also $A\ne K$, since $K$ is not reflexive, as shown in the proof of
Proposition~\ref{prop:class-S}. Consequently the even class of $A$ is a singleton. Since $A$ is reflexive, $\mathsf T(A)=A+K$ is defined and belongs to that same even class. Hence $A+K=A$, which gives $M+K=M$. Thus $S$ is almost symmetric, proving (2)$\Rightarrow$(1).
\end{proof}

In particular, every mixed component over an irreducible numerical semigroup has at most two vertices, since such a semigroup is symmetric or pseudosymmetric and hence almost symmetric; see \cite{RG09}.

\section{Semigroups of type two}\label{sec:type-two}

For semigroups of type two, we describe the mixed liaison classes and
show that the even-class bound $r(K)$ is attained.

Write $\operatorname{PF}(S)=\{F-d,F\}$ with $d>0$. If $d\in S$,
then the pseudo-Frobenius property of $F-d$ gives
$F=(F-d)+d\in S$, a contradiction. By
\cite[Corollary~4.11]{CDG24}, the normalized canonical ideal is
$K=\{0,d\}+S=S\cup(d+S)$. Proposition~4.14 of \cite{CDG24} gives
\begin{equation}\label{eq:type-two-reduction}
r(K)=\min\{n\ge1:(n+1)d\in S\}.
\end{equation}
Thus $d,2d,\ldots,r(K)d\notin S$, whereas $(r(K)+1)d\in S$.

\subsection{Canonical addition and blocks}

Since $K=S\cup(d+S)$, every relative ideal $I$ of $S$ satisfies
\begin{equation}\label{eq:type-two-identities}
\begin{aligned}
I+K&=I\cup(I+d),\\
I-K&=I\cap(I-d),\\
I^{**}&=I\cup\bigl((I+d)\cap(I-d)\bigr).
\end{aligned}
\end{equation}
The first identity follows from $I+S=I$. By $S$-stability,
$z+K\subseteq I$ is equivalent to $z,z+d\in I$, which gives
the second identity. The third follows from \eqref{eq:closure}
and the first two.

For a normalized ideal $I$, read membership along each progression
$r+d\mathbb Z$, $0\le r<d$, in increasing order. Each progression has
an infinite initial block outside $I$ and an infinite final block in
$I$. A \emph{finite occupied block} is a finite maximal run of terms
in $I$. A \emph{bounded gap block} is a run of terms outside $I$
flanked by terms in $I$. Block lengths count terms of the progression.

By \eqref{eq:type-two-identities}, $I+K\ne I$ exactly when some term
in $I$ is immediately followed by a term outside $I$. This gives a
finite occupied block followed by a bounded gap block. In this case,
let $a(I)$ and $b(I)$ be the respective minimum lengths of these
blocks over all progressions. If $I+K=I$, the even class of $I$ is a
singleton by Proposition~\ref{prop:partial-dynamics}.

\begin{lemma}\label{lem:type-two-block-steps}
Let $I$ be normalized and not $K$-stable.
\begin{enumerate}
\item[(a)] The ideal $I$ is $S$-reflexive, equivalently $\mathsf T(I)$ is defined, if and only if $b(I)\ge2$. In that case, $a(\mathsf T(I))=a(I)+1$ and $b(\mathsf T(I))=b(I)-1$.
\item[(b)] The inverse step $\mathsf T^{-1}(I)$ is defined if and only if $a(I)\ge2$. In that case, $\mathsf T^{-1}(I)=I-K$ is already normalized, and $a(\mathsf T^{-1}(I))=a(I)-1$, $b(\mathsf T^{-1}(I))=b(I)+1$.
\end{enumerate}
\end{lemma}

\begin{proof}
\noindent\textup{(a)} By \eqref{eq:type-two-identities}, reflexive closure fills
exactly the bounded gap blocks of length one. Thus $I$ is
$S$-reflexive if and only if $b(I)\ge2$. In this case,
$I+K=I\cup(I+d)$ extends each finite occupied block by one term
to the right and shortens each bounded gap block by one.
All bounded gaps remain nonempty, so the occupied blocks do not
merge. Taking minima gives the stated formulas.

\medskip
\noindent\textup{(b)} Put $J=I-K=I\cap(I-d)$. By \eqref{eq:mixed-identities},
$J=S-(K-I)$ is an $S$-dual and hence $S$-reflexive. Moreover,
\[
J+K=I\cap\bigl((I-d)\cup(I+d)\bigr),
\]
so $J+K$ is obtained from $I$ by removing its isolated occupied
terms. If $\mathsf T^{-1}(I)$ is defined, \eqref{eq:closure} gives
$\mathsf T^{-1}(I)=J$. Hence $J+K=I$, which implies $a(I)\ge2$.

Conversely, if $a(I)\ge2$, then $J+K=I$. Since $0\in I$ and
$-d\notin I$, the absence of isolated occupied terms gives $d\in I$.
Thus $0\in J$, so $J$ is normalized and $\mathsf T(J)=I$.
Passing from $I$ to $J$ removes the last term of each finite
occupied block and adds one term to each bounded gap block.
All finite occupied blocks remain nonempty, so the bounded gaps
do not merge. Taking minima gives the stated formulas.
\end{proof}

\begin{theorem}\label{thm:type-two-classes}
Let $I$ be a normalized ideal of a type-two numerical semigroup
with $I+K\ne I$. Its mixed even liaison class is
\[
\mathcal E(I)=\{I-jK:1\le j<a(I)\}\cup\{I+jK:0\le j<b(I)\}.
\]
All the displayed ideals are normalized, and
$|\mathcal E(I)|=a(I)+b(I)-1$.

The mixed component of $I$ is
$\mathcal E(I)\cup\mathsf C(\mathcal E(I))$. It has
$|\mathcal E(I)|$ vertices if the two even classes coincide
and $2|\mathcal E(I)|$ otherwise.
\end{theorem}

\begin{proof}
By Lemma~\ref{lem:type-two-block-steps}, exactly $a(I)-1$ inverse
steps and $b(I)-1$ forward steps are possible from $I$. By
Lemma~\ref{lem:colon-basic}(d), $j$ inverse steps give $I-jK$.
Similarly, $j$ forward steps give $I+jK$. All these iterates
are normalized. By Proposition~\ref{prop:partial-dynamics},
they form the strict chain that constitutes the even class of $I$.
Counting its vertices gives $|\mathcal E(I)|=a(I)+b(I)-1$.

The component description and its size follow from
Theorem~\ref{thm:paths}.
\end{proof}

In Example~\ref{ex:sharp-bound}, we have $d=2$. For $A_2$, the only
finite occupied block is $\{0\}$, and the only bounded gap block is
$\{2,4,6\}$. Thus $(a(A_2),b(A_2))=(1,3)$, and
Theorem~\ref{thm:type-two-classes} gives $\mathcal E(A_2)=\{A_2,A_2+K,A_2+2K\}=\{A_2,A_1,A_0\}$.

\subsection{Attainment of the even-class bound}\label{sec:type-two-attainment}

\begin{proof}[Proof of Theorem~\ref{thm:type-two-attainment}]
Corollary~\ref{cor:component-bound} gives the upper bound $r(K)$. To prove attainment, note that $0\in H$, since $jd\notin S$ for $1\le j\le r(K)$, and that every integer greater than $r(K)d$ belongs to $H$. If $x\in H$ and $s\in S$, then $x+s$ cannot equal $jd-t$ with $t\in S$, since this would give $x=jd-(s+t)\notin H$. Hence $H+S\subseteq H$, so $H$ is a normalized relative ideal.

By \eqref{eq:type-two-attaining-ideal}, the complement of $H$ in
$\mathbb Z$ consists of the negative integers and the sets
$\{d-s,2d-s,\ldots,r(K)d-s\}$ for $s\in S$. A bounded gap block
contains only nonnegative terms, so it meets one of these sets.
By maximality it contains that entire set, which consists of
$r(K)$ consecutive terms in the same progression. Thus every
bounded gap block has length at least $r(K)$.

In $d\mathbb Z$, the terms $d,2d,\ldots,r(K)d$ form a bounded gap
block between $0$ and $(r(K)+1)d$, and $\{0\}$ is a finite occupied
block. Hence $a(H)=1$ and $b(H)=r(K)$. Since $d\in(H+K)\setminus H$,
Theorem~\ref{thm:type-two-classes} applies and gives the asserted
even class and its cardinality.
\end{proof}

The equality in Theorem~\ref{thm:type-two-attainment} can fail
in type three.

\begin{example}\label{ex:strict-bound}
Let $S=\langle7,9,10,11\rangle$. We have
$S=\{0,7,9,10,11,14\}\cup(16+\mathbb N)$ and
$\operatorname{PF}(S)=\{12,13,15\}$, so $S$ has type three and $F=15$.
Moreover,
\[
\begin{aligned}
K&=\{0,2,3\}+S=\mathbb N\setminus\{1,4,5,6,8,15\},\\
2K&=\mathbb N\setminus\{1,8\},\qquad
3K=4K=\mathbb N\setminus\{1\}.
\end{aligned}
\]
Thus $r(K)=3$.

If a mixed even class had at least three elements, then
Proposition~\ref{prop:partial-dynamics} would give three consecutive
members $I\subsetneq J\subsetneq J+K$, with $J=I+K$ and $J$
$S$-reflexive. Since $0\in I$, we have $K\subseteq J$. Putting
$D=S-J$, we obtain
\[
16+\mathbb N\subseteq D\subseteq S-K
=\{7,14\}\cup(16+\mathbb N).
\]
Since $D+S\subseteq D$, membership of $7$ forces membership of $14$.
The possible values of $D$ and their residuals $J=S-D$ are
\[
\begin{array}{c|c}
D & J=S-D\\ \hline
16+\mathbb N & \mathbb N\\
\{14\}\cup(16+\mathbb N) & \mathbb N\setminus\{1\}\\
\{7,14\}\cup(16+\mathbb N) & V:=\mathbb N\setminus\{1,5,6,8\}.
\end{array}
\]

The first two ideals are $K$-stable, so $J$ would have to be $V$.
But if $I+K=V$, then $1\notin I$ because $I\subseteq V$.
Also $2\notin I$ because $2+3=5\notin V$, and $4\notin I$ because
$4+2=6\notin V$. Since $K\cap[0,4]=\{0,2,3\}$, these exclusions
give $4\notin I+K$, contrary to $4\in V$.

Thus every mixed even class has at most two elements.
Proposition~\ref{prop:class-S} gives the even class $\{S,K\}$, so
$\max_{\mathcal E}|\mathcal E|=2<r(K)$. By Theorem~\ref{thm:paths},
every mixed component has at most four vertices, strictly fewer
than $2r(K)=6$.
\end{example}

\begin{remark}
The even-class bound can be attained while every mixed component
has fewer than $2r(K)$ vertices. For $S=\langle3,4,5\rangle$,
we have $K=\{0,1\}+S$ and $r(K)=2$. Since $M=3+\mathbb N$
satisfies $M+K=M$, the semigroup is almost symmetric. By
Proposition~\ref{prop:almost-symmetric}, every mixed component
has at most two vertices. By Proposition~\ref{prop:class-S},
$\{S,K\}$ is a single mixed even liaison class of size $r(K)$.
\end{remark}

\subsection{Unbounded families}

\begin{theorem}\label{thm:loop-family}
Let $S$ be a numerical semigroup of type two with $d=1$ and
multiplicity $m$. Then $r(K)=m-1$. Put
\[
H_j=\{0,1,\ldots,j\}\cup(m+\mathbb N)\qquad(0\le j\le m-2).
\]
The mixed component of $H_0$ is $\{H_0,\ldots,H_{m-2}\}$.
Its underlying simple graph is a path with $m-1$ vertices, all in
the same even class, and its only nonreflexive vertex is $H_{m-2}$.
It has exactly one loop. The loop is principal at $H_{(m-3)/2}$
when $m$ is odd and canonical at $H_{(m-2)/2}$ when $m$ is even.

For every odd $m\ge5$, the semigroup $T_m=\langle m,m+1,m+2\rangle$
has embedding dimension three and satisfies these hypotheses.
\end{theorem}

\begin{proof}
For $d=1$, formula~\eqref{eq:type-two-reduction} gives $r(K)=m-1$.
In the construction of Theorem~\ref{thm:type-two-attainment},
$j-s<0$ whenever $1\le j\le m-1$ and $s\in S\setminus\{0\}$.
Thus the ideal $H$ in that theorem is $H_0$. Since
$H_0+jK=H_j$ for $0\le j\le m-2$, its even class is
$\mathcal E=\{H_0,\ldots,H_{m-2}\}$.

By Lemma~\ref{lem:conductor-dual}, $\mathsf C(H_0)=H_{m-2}$.
Proposition~\ref{prop:partial-dynamics}(c) shows that $\mathsf C$
preserves $\mathcal E$ and reverses its order. By
Theorem~\ref{thm:paths}, $\mathcal E$ is the full component and its
underlying simple graph is a path. Since $\mathsf C(H_j)=H_{m-2-j}$
for $0\le j\le m-2$ and $\mathsf P(H_j)=\mathsf C(H_{j+1})=H_{m-3-j}$
for $0\le j\le m-3$, its unique loop is principal at $H_{(m-3)/2}$
when $m$ is odd and canonical at $H_{(m-2)/2}$ when $m$ is even.
By Proposition~\ref{prop:partial-dynamics}(c), the last ideal
$H_{m-2}$ is the only nonreflexive member.

The three generators of $T_m$ lie in $[m,2m)$, so they are minimal.
By \cite[Lemma~2.7]{Mat04}, $T_m$ has exactly two pseudo-Frobenius
numbers, and they are consecutive.
\end{proof}

We now construct a loopless component in which both even liaison
classes grow with the multiplicity.

For odd $m=2q+1\ge5$, put
\[
S_m=\langle m,m+2,m+4\rangle,\qquad
I_m=\mathbb N\setminus\{1,m+1\},
\]
and write $K=K(S_m)$.

\begin{theorem}\label{thm:unbounded}
The semigroup $S_m$ has embedding dimension three and type two.
The mixed component of $I_m$ is a loopless path on $m+1$ vertices
whose edges alternate between canonical and principal links.
It has two even liaison classes, each with $(m+1)/2$ vertices.
\end{theorem}

\begin{proof}
The three generators of $S_m$ lie in $[m,2m)$, so they are minimal.
By \cite[Lemma~2.7]{Mat04}, the two pseudo-Frobenius numbers of $S_m$
differ by $2$. Thus $K=\{0,2\}+S_m$ by
\cite[Corollary~4.11]{CDG24}.

For $0\le j\le q$, put
\[
B_j=\{0,2,\ldots,2j,m\}\cup(m+2+\mathbb N).
\]
The set $B_0$ contains $S_m$, and adding a positive element of $S_m$
to a positive element of $B_0$ gives an integer at least $2m$.
Thus $B_0$ is a normalized relative ideal. In the even progression,
its singleton occupied block $\{0\}$ is followed by the gap block
$2,4,\ldots,m+1$ of length $q+1$. In the odd progression, the only
occupied block starts at $m$ and is infinite. Hence $a(B_0)=1$ and
$b(B_0)=q+1$. By \eqref{eq:type-two-identities}, successive unions
with the translate by $2$ give $B_0+jK=B_j$ for $0\le j\le q$.
Theorem~\ref{thm:type-two-classes} therefore gives
$\mathcal E(B_0)=\{B_0,\ldots,B_q\}$.

Each $B_j$ has Frobenius number $m+1$. By
Lemma~\ref{lem:conductor-dual}, its normalized canonical dual is
\[
A_j=\mathsf C(B_j)
=\mathbb N\setminus\bigl(\{1\}\cup\{m+1-2r:0\le r\le j\}\bigr).
\]
In particular, $A_0=I_m$. The two even classes are disjoint because
$3\in A_j$ and $3\notin B_j$ for every $j$. By
Theorem~\ref{thm:paths}, their union is the full component and it
has no loops. For $0\le j<q$, the step $\mathsf T(B_j)=B_{j+1}$
gives $\mathsf P(B_j)=\mathsf C(B_{j+1})=A_{j+1}$. 
\end{proof}

\begin{remark}
The least positive even element of $S_m$ is $2m$, since all three
generators are odd. With $d=2$, \eqref{eq:type-two-reduction} gives
$r(K)=m-1$. The two even classes of the component above have
$(m+1)/2<m-1$ vertices, while
Theorem~\ref{thm:type-two-attainment} gives an even class with
$m-1$ vertices over the same semigroup $S_m$.
\end{remark}

\medskip
\noindent\textbf{Acknowledgements.}
The author thanks Prof.\ S.\ Estrada for pointing out notational and conceptual inaccuracies in the subsection on canonical links in a previous version.

\section*{Statements and Declarations}

\noindent\textbf{Funding.}
The author is partially supported by project PID2022-138906NB-C21, funded by \mbox{MCIN/AEI}/\allowbreak10.13039/\allowbreak501100011033 and the European Union NextGenerationEU/PRTR, and by grant GR24068, funded by the Junta de Extremadura and co-funded by the European Regional Development Fund (ERDF).

\medskip
\noindent\textbf{Competing interests.}
The author has no relevant financial or non-financial interests to disclose.

\medskip
\noindent\textbf{Data availability.}
All mathematical constructions and examples supporting the results are included in the article. No external datasets are required.


\end{document}